\documentclass[article,12pt]{amsart}
\usepackage{amsfonts}
\usepackage{amsmath}
\usepackage{graphicx}
\usepackage{color}
\usepackage{epsfig}

\newcommand{\dC}{\mathbb{C}}

\newcommand{\dR}{\mathbb{R}}
\newcommand{\dE}{\mathbb{E}}
\newcommand{\dL}{\mathbb{L}}
\newcommand{\dP}{\mathbb{P}}

\newcommand{\cB}{\mathcal{B}}

\newcommand{\cF}{\mathcal{F}}
\newcommand{\cG}{\mathcal{G}}

\newcommand{\cN}{\mathcal{N}}
\newcommand{\cL}{\mathcal{L}}
\newcommand{\cR}{\mathcal{R}}

\newcommand{\ind}{\mbox{1}\kern-.25em \mbox{I}}

\def\build#1_#2^#3{\mathrel{\mathop{\kern 0pt#1}\limits_{#2}^{#3}}}

\def\videbox{\mathbin{\vbox{\hrule\hbox{\vrule height1ex \kern.5em
\vrule height1ex}\hrule}}}
\def\demend{\hfill $\videbox$\\}

\numberwithin{equation}{section}
\theoremstyle{plain}
\newtheorem{lem}{Lemma}[section]
\newtheorem{thm}{Theorem}[section]
\newtheorem{rem}{Remark}[section]

\newtheorem{defi}{Definition}[section]

\email{bernard.bercu@math.u-bordeaux.fr}
\email{fabrice.gamboa@math.univ-toulouse.fr}

\begin{document}
\title[On random Binomial-Fibonacci sequences]
{On random Binomial-Fibonacci sequences}
\author{Bernard Bercu}
\address{Universit\'e de Bordeaux, 
Institut de Math\'ematiques de Bordeaux,
UMR CNRS 5251, 351 cours de la lib\'eration, 
33405 Talence cedex, France.}

\author{Fabrice Gamboa}
\address{Universit\'e de Toulouse, Institut de Math\'ematiques de Toulouse, 31062 Toulouse Cedex 9, France.}

\begin{abstract}
Motivated by the pioneer work of Viswanath \cite{VIS2000} on random Fibonacci sequences and the development of random Euclidean addition chains
in cryptography, we introduce a new family of random Binomial-Fibonacci sequences. 
We study the asymptotic behavior of their finite-time Lyapunov exponents. In the special case of random Bernoulli-Fibonacci sequences, we
show that they convergence almost surely to an explicit Lyapunov exponent and prove a central limit theorem with an explicit formula for the variance. 
We also show that they satisfy a large deviation principle. 
\end{abstract}


\maketitle

\vspace{-4ex}
\section{Introduction and Motivation}
\label{SectionIM}
Growth and decay of Fibonacci numbers is a very attractive field of research in arithmetic, combinatorics
and probability. While a wide literature is available on the beautiful world of Fibonacci sequences \cite{KOS2001}, 
very few references may be found on random Fibonacci sequences apart the remarkable papers of Viswanath \cite{VIS2000}, of Embree
and Trefethen \cite{ET1999}, and of Janvresse {\it et al} \cite{JRR2008, JRR2009, JRR2010}. 
Our motivation for investigating random Fibonacci sequences concerns two areas
of probability and cryptography. Fist of all, our goal is to introduce a new family of random Bernoulli-Fibonacci sequences for which 
it is possible to carry out a thorough analysis of its asymptotic behavior. Next, we hope it would be possible to make use of
our new family of random Bernoulli-Fibonacci sequences for generating random keys in cryptography 
\par
Our first motivation comes from the study of the asymptoptic behavior of random Fibonacci sequences.
They are given by the initial states $X_0=1, X_1=1$ and for all $n\geq 2$, by the recurrence relation
\begin{equation}
\label{RANDOMFIB}
X_{n+1}= \alpha \theta_{n+1} X_{n}+ \beta \xi_{n+1} X_{n-1}
\end{equation}
where $\alpha, \beta$ are positive real numbers and $(\theta_n)$ and $(\xi_n)$ are two independent sequences
of independent and identically distributed random variables with Rademacher $\cR(p)$ and $\cR(q)$ distributions
with $ p,q \in \,]0,1[$, which means that
\begin{equation*}
\theta_{n+1}= \left \{ \begin{array}{lll}
    \,\,1 & \text{ with proba. } \hspace{2ex} p, \vspace{1ex}\\
    \!\!-1 & \text{ with proba. } \hspace{2ex} 1-p,  \end{array}  \right.
    \qquad
\xi_{n+1}= \left \{ \begin{array}{lll}
    \,\,1& \text{ with proba. } \hspace{2ex} q, \vspace{1ex}\\
    \!\!-1 & \text{ with proba. } \hspace{2ex} 1-q.  \end{array}  \right.    
\end{equation*}
The pioneer work of Viswanath \cite{VIS2000}
was devoted to 
$\alpha=1, \beta=1$, $p=1/2, q=1$, 
\begin{equation}
\label{RANDOMVIS}
X_{n+1}= \theta_{n+1} X_{n}+ X_{n-1}.
\end{equation}
It was shown by Furstenberg and Kesten \cite{F1963, FK1960} that
\vspace{-1ex}
\begin{equation}
\label{LIMVIS}
\lim_{n\rightarrow \infty} \frac{1}{n}  \log|X_{n}|= \gamma \hspace{1cm} \text{a.s.}
\vspace{-1ex}
\end{equation}
where the Lyapunov exponent $\gamma$ was later called Viswanath's constant. The determination of the limiting
value $\gamma$ relies on Furstenberg's formula, 
\begin{equation}
\label{FURFORM}
\gamma=\frac{1}{2}\int_0^\infty \log \Bigl( \frac{1+4x^4}{(1+x^2)^2} \Bigr) \,d\nu(x)
\end{equation}
where $\nu$ is a rather complicated fractal probability measure defined inductively on
Stern-Brocot intervals. 
Some extension for $\alpha=1, \beta=1$, and $p\neq 1/2$ may be found
in \cite{JRR2008} and for $\alpha >0$ in \cite{JRR2009, JRR2010, SIR2001}. 
Viswanath was only able to calculate eight decimal places of $\gamma$, and it is quite difficult to go any further 
in calculating the remaining decimal places \cite{BAI2007, BAT2002}. 
McLellan \cite{MCL2013} tried to shed light on certain properties of Viswanath's constant, but as she pointed out, we do not even know if 
$\gamma$ is irrational or transcendental. Here,
our primary goal is to introduce a new family of random Bernoulli-Fibonacci sequences for which it is possible 
to provide an explicit and easily computable form of their asymptotic Lyapunov exponent.
\par
Our second motivation for studying such sequences is related to Euclidean addition chains  \cite{BRA1939}, \cite{MT2011}
which were introduced to generate random cryptographic keys that are resistant to power analysis attacks.
An Euclidean addition chain of length $\ell$, leading
to an integer $a$, is a sequence $(\theta_1, \theta_2, \ldots, \theta_\ell) \in \{0,1\}^\ell$ associated with
$(X_1, X_2, \ldots, X_\ell)$ such that the initial states $X_0=X_1=1, X_2=2$ and $\theta_1=1, \theta_2=2$,
the final state $X_\ell\!=\!a$, and the intermediate states are given, for $2 \leq n \leq \ell-1$, by 
\begin{equation}
\label{EAC}
X_{n+1}=  \theta_{n+1}(X_{n}+X_{n-1})+(1- \theta_{n+1}) (X_{n}+X_{k} )
\end{equation}
where $k<n-1$ is such that $X_{n}=X_{n-1}+X_{k}$. The case $\theta_{n+1}=1$ corresponds to the big step 
because we add the biggest of the two possible integers to $X_{n}$, while the case $\theta_{n+1}=0$ is the small step
as we add the smallest one \cite{HV2010}, \cite{MEL2007}. For example, from the Euclidean addition chain 
$(1, 1, 0,0,1,0,1,1,1,0)$ of length $\ell=10$, one can compute the integer $a=57$ via
$(1,2,3,4,7,9,16,25,41,57)$.
Several algorithms to encrypt and decrypt a message by use of random Euclidean addition chains were
recently proposed by Herbaut {\it et al}  \cite{HLMTV2010, HV2010}. Hence, the study of the asymptotic behavior of the
random Bernoulli-Fibonacci sequences will enlighten the family of random Euclidean addition chains
in cryptography.
\par
The paper is organized as follows. Section 2 is devoted to the introduction 
of the new family of random Binomial-Fibonacci sequences. In Section 3, 
we show the almost sure convergence of their finite-time Lyapunov exponents and prove a central limit theorem.
In the special case of random Bernoulli-Fibonacci sequences, we provide an explicit formula for the Lyapunov exponent
as well as for the asymptotic variance.
We also show that their finite-time Lyapunov exponents satisfy a large deviation principle. 
All technical proofs are postponed in the Appendices.

\section{Random Binomial-Fibonacci sequences}
\label{SectionBF}

\begin{defi}
We shall say that $(X_n)$ is random Binomial-Fibonacci $\cB\cF(N,p)$ sequence with parameters $N\geq 1$
and $0< p< 1$ if $X_0=1, X_1=1$ and for all $1\leq n \leq N$, $X_{n+1}= X_{n}+ X_{n-1}$, whereas
for all $n \geq N+1$,
\begin{equation}
\label{BFRAND}
X_{n+1}=  X_{n}+ X_{n+\theta_{n+1}-N}
\end{equation}
where $\theta_{n+1}$ stands for a Binomial $\cB(N,p)$ random variable, independent of $X_{1},\ldots, X_{n}$,
which means that $\theta_{n+1}$ takes the values $\{0,1,\ldots,N\}$ and for all $0\leq k\leq N$, 
\begin{equation*}
\dP(\theta_{n+1}=k) =C_N^k p^k (1-p)^{N-k}.
\end{equation*}
\end{defi}
\par
For example, the random Bernoulli-Fibonacci $\cB\cF(1,p)=\cB\cF(p)$ sequence is 
given by $X_0=1, X_1=1, X_2=2$ and for all $n \geq 2$,
\begin{equation}
\label{BERFRAND1}
X_{n+1}=  X_{n}+ X_{n+\theta_{n+1}-1}
\end{equation}
where $\theta_{n+1}$ is a Bernoulli random variable with $\cB(p)$ distribution. On can observe that \eqref{BERFRAND1} can be rewritten,
for all $n \geq 2$, as
\begin{equation}
\label{BERFRAND2}
X_{n+1}=  (1+\theta_{n+1}) X_{n}+ (1-\theta_{n+1}) X_{n-1}.
\end{equation}
The crucial difference between \eqref{EAC} or \eqref{RANDOMFIB} and \eqref{BERFRAND2}  is that
we allow super big step which means that, with probability $p$, we could add $X_n$ to itself.
Another example is the random Binomial-Fibonacci $\cB\cF(2,p)$ sequence
given by $X_0=1, X_1=1, X_2=2$, $X_3=3$ and for all $n \geq 3$,
\begin{equation}
\label{BINFRAND2-1}
X_{n+1}=  X_{n}+ X_{n+\theta_{n+1}-2}
\end{equation}
where $\theta_{n+1}$ is a Binomial random variable with $\cB(2,p)$ distribution. One can easily see from \eqref{BINFRAND2-1}  
that for all $n \geq 3$, 
\begin{align}
\label{BINFRAND2-2}
X_{n+1} &= \Bigl(1+\frac{1}{2}\theta_{n+1}(\theta_{n+1}-1)\Bigr) X_{n}+\theta_{n+1}(2-\theta_{n+1}) X_{n-1} \nonumber \\
&+ \frac{1}{2} (1-\theta_{n+1})(2-\theta_{n+1}) X_{n-2}.
\end{align}

\noindent
The key point in our analysis of the asymptotic behavior of $(X_n)$ 
is to make use of the Riccati ratio $R_n$ defined, for all $n\geq 1$, by
\begin{equation}
\label{RICCATI}
R_{n}=  \frac{X_{n}}{X_{n-1}}.
\end{equation}
Actually, the finite-time Lyapunov exponent associated with $(X_n)$ is given by
\begin{equation}
\label{EQLOG}
\frac{1}{n}\log (X_{n})= \frac{1}{n}\sum_{k=1}^n \log (R_{k}).
\end{equation}
Our goal is to provide a complete analysis of the asymptotic behavior of the finite-time Lyapunov exponent associated with
Binomial-Fibonacci $\cB\cF(N,p)$ sequences with as special focus on Bernoulli-Fibonacci $\cB\cF(p)$ sequences.

\section{Main results}
\label{SectionMR}

Our first result is devoted to the almost sure convergence of the finite-time Lyapunov exponent of random Binomial-Fibonacci $\cB\cF(N,p)$ sequences. In all the sequel, $(F_k)$ is the classical Fibonacci sequence and $G$ stands for a random variable with Geometric $\cG(p)$ distribution, that is for all $k\geq 0$, $\dP(G=k) = p (1-p)^{k}$.
\begin{thm}
\label{T-ASCVGBERF}
Assume that $(X_n)$ is a random Binomial-Fibonacci $\cB\cF(N,p)$ sequence with parameters $N \geq 1$ and $0< p < 1$. Then, we have 
\begin{equation}
\label{BERFASCVG}
\lim_{n\rightarrow \infty} \frac{1}{n}  \log (X_{n})= L(p) \hspace{1cm} \text{a.s.}
\end{equation}
where the Lyapunov exponent $L(p)$ is finite. This convergence also holds in $\dL^1$. Moreover, if $(X_n)$ is a 
Bernoulli-Fibonacci $\cB\cF(p)$ sequence, $L(p)$ is given by
\begin{equation}
\label{LYAP}
L(p)=\dE\left[\log \!\left( \frac{F_{G+3}}{F_{G+2}} \right)\right]=\sum_{k=0}^\infty p (1-p)^{k} \log \!\left( \frac{F_{k+3}}{F_{k+2}} \right).
\end{equation}
\end{thm}
\noindent{\bf Proof.}
The proof is given in Appendix\,A. \demend
\vspace{-3ex}
\begin{rem} 
\label{R-ASCVG}
Denote by $\varphi$ the golden ratio, $\psi=-\varphi^{-1}$ and $r=\psi \varphi^{-1}=-\varphi^{-2}$.
It follows from Binet's formula \cite{KOS2001} that for all $k\geq 0$,
$$
F_k=\frac{\varphi^k - \psi^k}{\sqrt{5}}
\qquad \text{and} \qquad
\frac{F_{k+3}}{F_{k+2}}=\varphi \left(\frac{1-r^{k+3}}{1-r^{k+2}}\right).
$$
Hence, if $(X_n)$ is a Bernoulli-Fibonacci $\cB\cF(p)$ sequence, $L(p)$ can be rewritten as
\begin{equation*}
L(p)=\log(\varphi) + \sum_{k=1}^\infty \frac{pr^{2k} (1-r^{k})}{k\big(1-(1-p)r^k\big)}.
\end{equation*}
Since $r^2<0.146$, this series converges exponentially fast, which means that we can compute the value of $L(p)$ very precisely. However,
as for Viswanath's constant $\gamma$, we do not know whether $L(p)$ is irrational or transcendental.
\end{rem}

\begin{figure}[ht]
\vspace{-4.4cm}
\begin{center}
\includegraphics[scale=0.44]{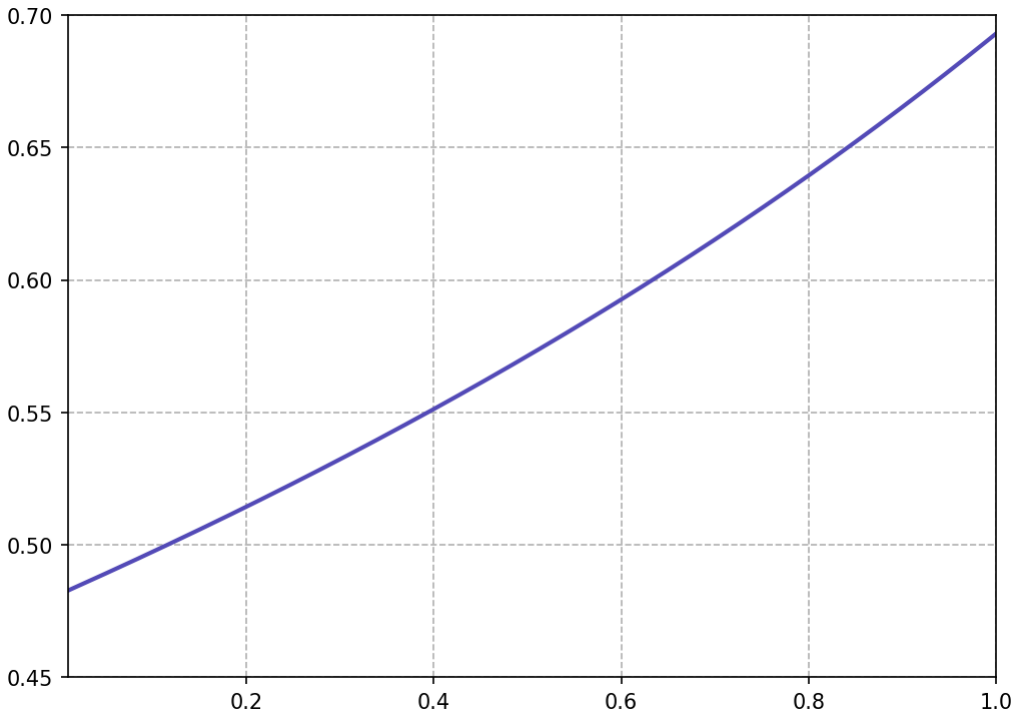}
\vspace{-4.2cm}
\caption{Variation of the Lyapunov exponent $L(p)$ given by \eqref{LYAP} with respect to $p$, going from $L(0)=\log(\varphi)$ to $L(1)=\log(2)$.}
\label{Fig-BFLp}
\end{center}
\end{figure}

Our next result concerns the central limit theorem for the finite-time Lyapunov exponent of random Binomial-Fibonacci $\cB\cF(N,p)$ sequences.

\begin{thm}
\label{T-ANBERF}
Assume that $(X_n)$ is a random Binomial-Fibonacci $\cB\cF(N,p)$ sequence with parameters $N \geq 1$ and $0< p < 1$. Then, we have the central limit theorem
\begin{equation}
\label{BERFAN}
\frac{\log (X_{n}) - n L(p)}{\sqrt{n}}   \underset{n\rightarrow+\infty}{\overset{\cL}{\rightarrow}} \cN\big(0, \Gamma(p)\big).
\end{equation}
Moreover, if $(X_n)$ is a 
Bernoulli-Fibonacci $\cB\cF(p)$ sequence, the asymptotic variance $\Gamma(p)$ is given by $\Gamma(p)=\sigma^2(p)+2\tau^2(p)$ with
\begin{equation}
\label{AVAR1}
\sigma^2(p)=\text{Var}\left(\!\log \!\left( \frac{F_{G+3}}{F_{G+2}} \right)\!\right)=\sum_{k=0}^\infty p (1-p)^{k} \log \!\left( \frac{F_{k+3}}{F_{k+2}} \right)^2 - L^2(p)
\end{equation}
and
\begin{equation}
\label{AVAR2}
\tau^2(p)=\sum_{\ell=1}^\infty \text{Cov}\left(\!\log \!\left( \frac{F_{G+3}}{F_{G+2}} \right) \!, \log \!\left( \frac{F_{G+\ell+3}}{F_{G+\ell+2}} \right)
\!\right).
\end{equation}
\end{thm}

\noindent{\bf Proof.}
The proof is given in Appendix\,B. \demend

\begin{rem}
If $(X_n)$ is a Bernoulli-Fibonacci $\cB\cF(p)$ sequence, we can show after some tedious but straightforward calculations that $\Gamma(p)$ can be drastically reduced to
\begin{eqnarray}
\label{AVAR}
\Gamma(p) &=&p\dE\Big[\Big( \log(F_{G+3}) -(G+1)L(p) \Big)^2\Big] \notag \\
&=&\sum_{k=0}^\infty p^2 (1-p)^{k}  \left( \log(F_{k+3}) -(k+1)L(p) \right)^2.
\end{eqnarray}

\end{rem}

\begin{figure}[ht]
\vspace{-4.4cm}
\begin{center}
\includegraphics[scale=0.44]{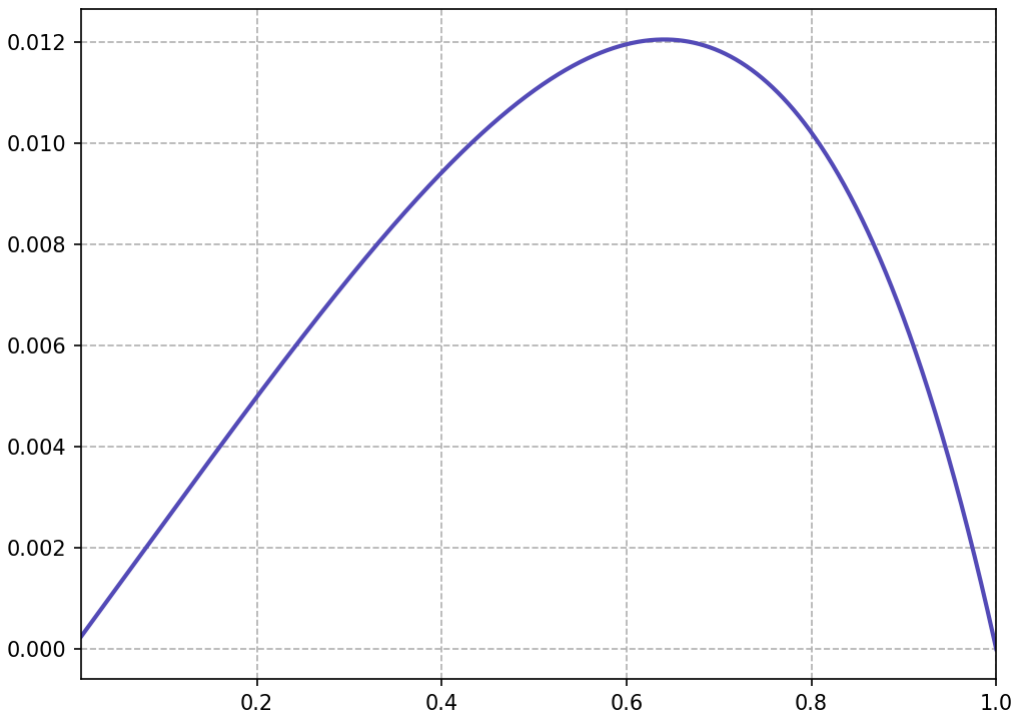}
\vspace{-4cm}
\caption{Variation of the asymptotic variance $\Gamma(p)$ given by \eqref{AVAR} with respect to $p$ with
$\Gamma(0)=\Gamma(1)=0$. t is worth noting that the asymptotic variance $\Gamma(p)$ is quite small.}
\label{Fig-BFCLT}
\end{center}
\end{figure}

Our last result deals with the large deviation principle for the finite-time Lyapunov exponent of random Bernoulli-Fibonacci $\cB\cF(p)$ sequences.

\begin{thm}
\label{T-LDPBERF}
Assume that $(X_n)$ is a random Bernoulli-Fibonacci $\cB\cF(p)$ sequence with $0< p < 1$. 
Then, the sequence $(\log X_n/n)$ satisfies a large deviation principle with speed $n$ and good rate function  
\begin{equation}
\label{RATEF}
I(p,x)=\sup_{t \in \dR} \Big\{ x t - \Lambda(p,t) \Big\}
\end{equation}
where, for any real number $t$, the limiting cumulant generating function $\Lambda(p,t)$ is the unique solution of the implicit equation
\begin{equation}
\label{DEFLAMBDA}
\dE\Big[ F^t_{G+3} \! \exp\big(\!-\!(G+1)\Lambda(p,t)\big) \Big] 
=1.
\end{equation}
\end{thm}

\noindent{\bf Proof.}
The proof is given in Appendix\,C. \demend

\begin{rem}
\label{R-LDP}
The effective domain of the rate function is $[\log(\varphi),\log(2)]$ with $I(p, \log(\varphi))\!=\!-\log(1-p)$, $I(p, \log(2))\!=\!-\log(p)$. In addition,
$\Lambda(p,t)$ is finite for all $t \in \dR$ and is infinitely differentiable. Let $H_G(p,t)=F^t_{G+3}\! \exp\big(\!-\!(G+1)\Lambda(p,t)\big)$. By deriving the implicit equation \eqref{DEFLAMBDA}, we obtain that for all $t \in \dR$,
\begin{equation}
\label{DERLAMBDA}
 \Lambda^\prime(p,t)=\frac{\dE\Big[ \log(F_{G+3}) H_G(p,t) \Big]}{\dE\Big[ (G+1) H_G(p,t) \Big]}.
\end{equation}
We find again as expected that $ \Lambda^\prime(p,0)=L(p)$. As a matter of fact, $\Lambda(p,0)=0$, 
$H_G(p,0)=1$, $\dE[G+1]=1/p$ and
\begin{align*}
 \dE[ \log(F_{G+3})] &=  \sum_{k=0}^\infty p(1-p)^k \log F_{k+3}=
 \sum_{k=0}^\infty p(1-p)^k \sum_{\ell=0}^k \log\left(\frac{F_{\ell+3}}{F_{\ell+2}}\right), \\
 &= \sum_{\ell=0}^\infty p \log\left(\frac{F_{\ell+3}}{F_{\ell+2}}\right) \sum_{k=\ell}^\infty (1-p)^k  
 = \frac{1}{p}\sum_{\ell=0}^\infty p(1-p)^\ell \log\left(\frac{F_{\ell+3}}{F_{\ell+2}}\right), \\
 &=\frac{1}{p}\dE\left[\log \left( \frac{F_{G+3}}{F_{G+2}} \right)\right]=\frac{1}{p}L(p).
\end{align*}
By the same token,  we also obtain that for all $t \in \dR$,
\begin{equation}
\label{DER2LAMBDA}
 \Lambda^{\prime\prime} (p,t)=\frac{\dE\Big[ \Big(\!\log(F_{G+3}) -(G+1)L^\prime(t) \Big)^2 H_G(p,t) \Big]}{\dE\Big[ (G+1) H_G(p,t) \Big]},
\end{equation}
which immediately leads via \eqref{AVAR} to $ \Lambda^{\prime \prime}(p,0)=\Gamma(p)$. Equation \eqref{DER2LAMBDA} also implies that 
$\Lambda(p,t)$ is strictly convex.
\end{rem}

\begin{figure}[ht]
\vspace{-4cm}
\begin{center}
\includegraphics[scale=0.45]{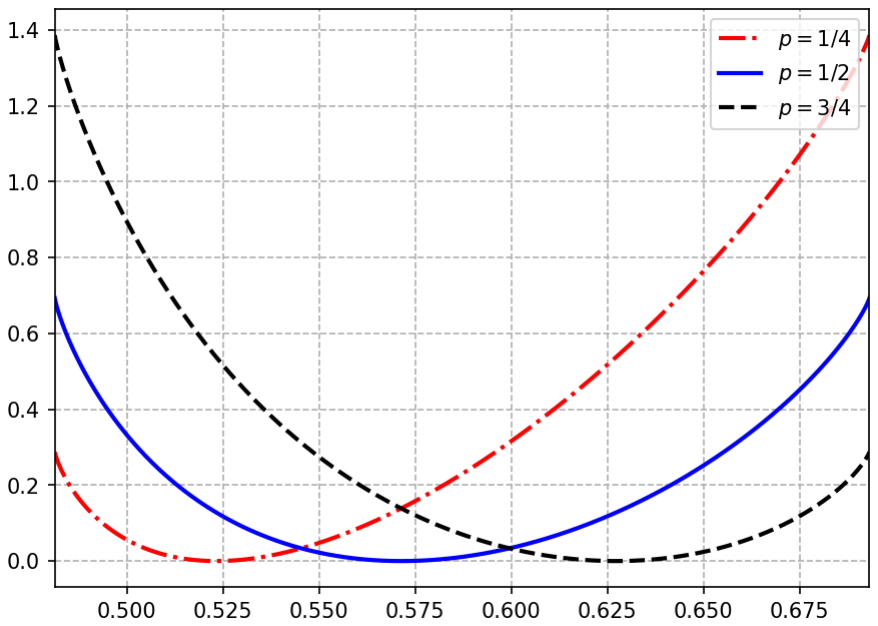}
\vspace{-4cm}
\caption{The rate function $I(p,x)$ for three different values of $p$.}
\label{Fig-BFLDP}
\end{center}
\end{figure}


\section{To go a little further}

We investigated the asymptotic behavior of random Binomial-Fibonacci sequences. An interesting challenge 
would be to identify the unique stationary distribution $\pi$ associated with the ratios of random Binomial-Fibonacci $\cB\cF(N,p)$ sequences with $N \geq 2$ and to provide an exact expression for the Lyapunov exponent 
$L(p)$ as in \eqref{LYAP}. It would also be interesting to study random Fibonacci sequences associated with 
discrete distribution $(\theta_n)$ with countably infinite support such as random Geometric-Fibonacci sequences
\begin{equation}
\label{INFFRAND}
X_{n+1}=  X_{n}+ X_{n+\theta_{n+1}-1\wedge 1}
\end{equation}
where $\theta_{n+1}$ stands for a Geometric $\cG(p)$ random variable, independent of $X_{1},\ldots, X_{n}$. In this context, it will not be possible to use directly the Furstenberg-Kesten theorem \cite{FK1960} on the norm of the product of random matrices.


\section*{Appendix A}

\begin{center}
{\small PROOF OF THE ALMOST SURE CONVERGENCE}
\end{center}

\renewcommand{\thesection}{\Alph{section}} 
\renewcommand{\theequation}
{\thesection.\arabic{equation}} \setcounter{section}{1}  
\setcounter{equation}{0}
Let $(\theta_{n})$ be a sequence of independent random variables sharing the same Binomial $\cB(N,p)$ distribution. Denote by $A_{\theta_n}$ the random square matrix of order $N+1$, 
\begin{equation*}
A_{\theta_n} = 
\begin{pmatrix}
   1 & 0 & \cdots & 1 & \cdots & 0 \\
   1 & 0 & \cdots  & \cdots & \cdots & 0 \\
   0 & 1 & 0  & \cdots & \cdots & 0 \\
   \vdots & \ddots &  \ddots & \ddots &  & \vdots \\
   \vdots &  & \ddots  & \ddots & \ddots & \vdots \\
   0 & \cdots & \cdots & 0 & 1 & 0 \\
\end{pmatrix}
\end{equation*}
where the second $1$ in the first row of $A_{\theta_n}$ is located at position $N - \theta_n +1$. One can observe that $\parallel \! A_{\theta_n} \!\parallel_\infty=2$, which obviously implies that 
$\dE\big[\log\big(\!\parallel \! A_{\theta_n} \! \parallel_\infty\! \big) \big]=\log(2)$. Consequently, we can deduce from  \cite[Theorem 2]{FK1960} that
\begin{equation}
\label{FKPROD}
\lim_{n \rightarrow \infty} \frac{1}{n} \log\big(\!\parallel A_{\theta_n} \cdots A_{1} \parallel_{\infty} \!\!\big)=L(p) 
\hspace{1cm} \text{a.s.}
\end{equation}
where the Lyapunov exponent $L(p)$ is always smaller than $\log(2)$.
We also obtain from \cite[Theorem 1]{FK1960} that convergence \eqref{FKPROD} also holds in $\dL^1$.
Hereafter, let $P_n$ be the product of random matrices defined by $P_{N+1}=I_{N+1}$ and, for all $n \geq N+2$, by $P_n=A_{\theta_n} \cdots  A_{\theta_{N+2}}$. It follows from \eqref{BFRAND} that for all $n \geq N+1$,
\begin{equation}
\label{BFRANDMATRIX}
Y_{n}=P_{n} Y_{N+1}
\end{equation}
where $Y_n$ stands for the random vector of $\dR^{N+1}$, $Y_n=(X_{n}, \ldots,X_{n-N})^T$. Since $(X_n)$ is an increasing sequence, $X_n < \parallel Y_n \parallel_{1} < (N+1)X_n$. In addition, we also have from \eqref{BFRANDMATRIX} that $\parallel P_n \parallel_\infty \leq \parallel Y_n \parallel_{1} \leq (N+1)F_{N+1}\parallel P_n \parallel_\infty$.  Then, \eqref{FKPROD} and \eqref{BFRANDMATRIX} lead to
\begin{equation*}
\lim_{n \rightarrow \infty} \frac{1}{n} \log(X_n)=\lim_{n \rightarrow \infty} \frac{1}{n} \log\big(\! \parallel Y_n \parallel_{1}\!\!\big)=\lim_{n \rightarrow \infty} \frac{1}{n} \log\big(\! \parallel P_n \parallel_{\infty}\!\!\big)=L(p) 
\hspace{1cm} \text{a.s.}
\end{equation*}
which is precisely \eqref{BERFASCVG}.
From now on, we focus our attention on the proof of \eqref{LYAP} for random Bernoulli-Fibonacci $\cB\cF(p)$ sequences. It follows from \eqref{BERFRAND2} and \eqref{RICCATI} 
that the sequence $(R_n)$ is given by the iterations of random Lipschitz functions,
\begin{equation}
\label{RANDITONE}
R_{n+1}=f_{\theta_{n+1}}(R_n)
\end{equation}
where $R_n \in [3/2,2]$, $\theta_{n+1}$ is a Bernoulli random variable with $\cB(p)$ distribution, and
$f_{\theta_{n+1}}$ is the Lipschitz 
function defined, for all $x \!\in \![3/2,2]$, by
$$f_{\theta_{n+1}}(x)=(1+\theta_{n+1}) + \frac{(1-\theta_{n+1})}{x}. $$
It is easy to see that for all $x,y\in [3/2,2]$,
$$ |f_{\theta_{n+1}}(x)-f_{\theta_{n+1}}(y)| \leq \frac{4}{9}K_{\theta_{n+1}} |x-y| $$
where $K_{\theta_{n+1}}=1-\theta_{n+1}$. We have $\dE[K_{\theta_{n+1}}]=1-p$, $\dE[\log K_{\theta_{n+1}}]=-\infty$, as well as
$$\dE[|f_{\theta_{n+1}}(2)-2|]=\frac{1}{2}\dE[1-\theta_{n+1}]=\frac{1}{2}(1-p)<1.$$ 
Therefore, we can deduce from \cite[Theorem 1.1]{DIA1999}
that $(R_n)$ is a geometrically ergodic Markov chain.
Denote by $\pi$ its unique stationary distribution. We also obtain from relation \eqref{RANDITONE} that $\pi$ is a discrete
measure whose support is given by all the finite continued fractions of $2$, that is
$f^{(0)}(2)=2$ and for all $k\geq 1$, $f^{(k)}(2)=[1;1,\ldots,1,2]$.
In addition, we have for all $k\geq 0$, 
$
\pi \{f^{(k)}(2)\}=p (1-p)^k.
$
Furthermore, it is well-known that for all $k \geq 0$, $f^{(k)}(2)=F_{k+3}/F_{k+2}$.
Finally, \eqref{EQLOG} implies that 
$$
\lim_{n\rightarrow \infty} \frac{1}{n}  \log (X_{n}) = \lim_{n\rightarrow \infty} \frac{1}{n} \sum_{k=1}^n \log (R_{k}) = \sum_{k=0}^\infty p (1-p)^{k} \log \!\left( \frac{F_{k+3}}{F_{k+2}} \right)=L(p) 
\qquad \text{a.s.}
$$
which completes the proof of Theorem \ref{T-ASCVGBERF}.
\demend


\section*{Appendix B}

\begin{center}
{\small PROOF OF THE CENTRAL LIMIT THEOREM}
\end{center}

\renewcommand{\thesection}{\Alph{section}} 
\renewcommand{\theequation}
{\thesection.\arabic{equation}} \setcounter{section}{2}  
\setcounter{equation}{0}

The proof of Theorem \ref{T-ANBERF} is a direct application of the central limit theorem for regenerative ergodic Markov chains \cite[Theorem 4.1]{LATALA2008}. For all $n \geq N+1$, denote by $\cR_n$ the random vector of $\dR^{N}$ associated with the ratios, $\cR_n=(\cR_{1,n}, \ldots,\cR_{N,n})^T$ where, for all $1\leq k \leq N$, $\cR_{k,n}=X_n/X_{n-k}$. 
Let $E_N$ be the compact set of $\dR^N$, $E_N\!=\!\prod_{k=1}^N \big[(3/2)^k,2^k\big]$. It follows from \eqref{BFRAND} that
$\cR_n \in E_N$ and, for all $n \geq N+1$,
\begin{equation}
\label{RANDIT}
\cR_{n+1}=f_{\theta_{n+1}}(\cR_n)
\end{equation}
where $\theta_{n+1}$ is a Binomial random variable with $\cB(N,p)$ distribution and
$$
f_{\theta_{n+1}}(x)=g_{\theta_{n+1}}(x)
\begin{pmatrix}
1 \\
x_1 \\
\vdots \\
x_{N-1}
\end{pmatrix} 
$$
where $g_{\theta_{n+1}}(x)=1+x_{N-\theta_{n+1}}^{-1}$ if $\theta_{n+1} \neq N$, while
$g_{\theta_{n+1}}(x)=2$ if $\theta_{n+1} = N$. Hence, we deduce from \eqref{RANDIT} that for every $x \in E_N$,
\begin{equation}
\label{REGEN}
f^N_N(x)=f_n(f_N(\cdots f_N(x))\cdots)=\begin{pmatrix}
2 \\
4 \\
\vdots \\
2^{N}
\end{pmatrix} 
.
\end{equation}
Equation \eqref{REGEN} simply means that the boundary point $r_N=(2,4,\ldots,2^N)^T$ is a regeneration state for the Markov chain $(R_n)$.
Since $\dP(\theta_{n+1}=N)=p^N$, the probability to reach $r_N$ is $\dP(\theta_{n+1}=N, \ldots,\theta_{n+N}=N)=p^{N^2}$.
Consequently, the Markov chain $(R_n)$ satisfies the $N$-step minorization condition \cite[Equation (4)]{LATALA2008} with $m=N$, $\varepsilon=p^{N^2}$ and 
$\nu_N=\delta_{r_N}$.
Furthermore, we have from \eqref{EQLOG} that
\begin{equation}
\label{REGENLOG}
\log(X_n)=\sum_{k=1}^n \log(R_{k})= \sum_{k=1}^N \log(R_{k}) + \sum_{k=N+1}^n h(\cR_k)
\end{equation}
where, for all $x \in E_N$, $h(x)=\log(x_1)$. Since $3/2 \leq x_1 \leq 2$ on $E_N$, the function $h$ is obviously bounded. Finally, we can apply 
\cite[Theorem 4.1]{LATALA2008} to conclude that
\begin{equation*}
\frac{\log (X_{n}) - n L(p)}{\sqrt{n}}   \underset{n\rightarrow+\infty}{\overset{\cL}{\rightarrow}} \cN\big(0, \Gamma(p)\big).
\end{equation*}
In the special case where $(X_n)$ is a random Bernoulli-Fibonacci $\cB\cF(p)$ sequence, the asymptotic variance $\Gamma(p)$ can be found in
\cite[Theorem 2.4]{BENDA1998} or \cite[Corollary 1]{WU2000}, which completes the proof of Theorem \ref{T-ANBERF}. \demend


\section*{Appendix C}

\begin{center}
{\small PROOF OF THE LARGE DEVIATION PRINCIPLE}
\end{center}

\renewcommand{\thesection}{\Alph{section}} 
\renewcommand{\theequation}
{\thesection.\arabic{equation}} \setcounter{section}{3}  
\setcounter{equation}{0}

We carry on with the proof of Theorem \ref{T-LDPBERF} where $(X_n)$ is a random Bernoulli-Fibonacci $\cB\cF(p)$ sequence.
Denote by $m_n$ the Laplace transform of $\log(X_n)$. We have from \eqref{EQLOG} that for all $t \in \dR$ and for all $n \geq 1$,
\begin{equation}
\label{DEFMNT}
m_n(t)= \dE\Big[\exp(t \log(X_n))\Big]= \dE\Big[\prod_{k=1}^n R_k^t\Big].
\vspace{-1ex}
\end{equation}
One can observe that $m_n(t)$ is finite for all $n \geq 1$ and for all $t\in \dR$ since $X_n \leq 2^n$.
We shall make use of the generating function of the sequence $(m_n(t))$ defined, for all $t \in \dR$ and for all $z \in \dC$, by
\vspace{-1ex}
\begin{equation}
\label{DEFGENF}
G(t,z)=\sum_{n=2}^\infty m_n(t)z^{n-2}.
\vspace{-1ex}
\end{equation}
The proof of Theorem \ref{T-LDPBERF} relies on the two following lemmas.

\begin{lem}
\label{L-LAPLACE}
We have for all $t \in \dR$ and for all $n \geq 3$,
\begin{equation}
\label{RECMNT}
m_n(t)= (1-p)^{n-2}F_{n+1}^t + \sum_{k=0}^{n-3} p(1-p)^k F_{k+3}^t m_{n-k-1}(t).
\end{equation}
\end{lem}

\begin{proof}
We recall from \eqref{RANDITONE} that for all 
$n \geq 3$,
\begin{equation}
\label{PRODR}
\prod_{k=1}^n R_k= 2\prod_{k=3}^nf_{\theta_{k}}(R_{k-1}).
\end{equation}
On the event $\{\theta_3=0, \ldots, \theta_n=0\}$, which occurs with probability $(1-p)^{n-2}$, the product \eqref{PRODR} reduces to
\begin{equation}
\label{PRODRZ}
\prod_{k=1}^n R_k= F_3 \times\frac{F_4}{F_3} \times \cdots \times \frac{F_{n+1}}{F_n}=F_{n+1}.
\end{equation}
On the complementary of this event, there is at least one value of $1$ in $\theta_3, \ldots, \theta_n$. Let $k \in \{0,\ldots,n-3 \}$ be the number of zeros after the last $1$. This means that
$\theta_{n-k}=1$, $\theta_{n-k+1}=0$, \ldots, $\theta_n=0$, which occurs with probability $p(1-p)^{k}$. As before, the contribution of this last block reduces to
\begin{equation}
\label{PRODRZZ}
\prod_{\ell=n-k}^n R_k= F_3 \times\frac{F_4}{F_3} \times \cdots \times \frac{F_{k+3}}{F_{k+2}}=F_{k+3}.
\end{equation}
By conditioning on the $n-k-1$ terms preceding the last $1$, we deduce from \eqref{DEFMNT} and \eqref{PRODR} together with \eqref{PRODRZ} and \eqref{PRODRZZ} that 
\begin{equation*}
m_n(t)=(1-p)^{n-2}F_{n+1}^t + \sum_{k=0}^{n-3} p(1-p)^k F_{k+3}^t m_{n-k-1}(t), 
\end{equation*}
since the average contribution of these remaining terms
is precisely $m_{n-k-1}(t)$.
\end{proof}

\begin{lem}
\label{L-GEN}
We have for all $t \in \dR$ and for all $z \in \dC$ such that $|z| < R_G(t)$,
\begin{equation}
\label{CALCULATIONGEN}
G(t,z)= \frac{\Phi(t,z)}{1-pz \Phi(t,z)}
\end{equation}
where $R_G(t)$ is the unique positive solution of the equation $p z \Phi(t,z)=1$ with
\begin{equation}
\label{DEFPHI}
\Phi(t,z)=\sum_{k=0}^\infty \big( (1-p)z \big)^k F_{k+3}^t.
\end{equation}
\end{lem}

\begin{proof}
We have from \eqref{DEFGENF} and \eqref{RECMNT} that for all $t \in \dR$ and for all $z \in \dC$,
\begin{equation}
\label{DECGENF1}
G(t,z) = m_{2}(t)+ \sum_{n=3}^\infty \!\big((1-p)z\big)^{n-2}F_{n+1}^t + \sum_{n=3}^\infty \sum_{k=0}^{n-3} p(1-p)^k F_{k+3}^t m_{n-k-1}(t) z^{n-2}.
\end{equation}
On the one hand,
\begin{equation}
\label{DECGENF2}
\sum_{n=3}^\infty \big((1-p)z\big)^{n-2}F_{n+1}^t = \sum_{n=1}^\infty \big((1-p)z\big)^{n}F_{n+3}^t= \Phi(t,z) - m_2(t).
\end{equation}
On the other hand,
\begin{align*}
\sum_{n=3}^\infty \sum_{k=0}^{n-3} p(1-p)^k F_{k+3}^t m_{n-k-1}(t) z^{n-2} & = \sum_{k=0}^\infty \sum_{n=k+3}^\infty p(1-p)^k F_{k+3}^t m_{n-k-1}(t) z^{n-2}, \\
&= \sum_{k=0}^\infty p(1-p)^k F_{k+3}^t \sum_{n=k+3}^\infty  m_{n-k-1}(t) z^{n-2}, \\
&= \sum_{k=0}^\infty p(1-p)^k F_{k+3}^t \sum_{\ell=2}^\infty  m_{\ell}(t) z^{k+\ell-1}, 
\end{align*}
which leads to 
\begin{equation}
\label{DECGENF3}
\sum_{n=3}^\infty \sum_{k=0}^{n-3} p(1-p)^k F_{k+3}^t m_{n-k-1}(t) z^{n-2}= p z \Phi(t,z) G(t,z).
\end{equation}
Therefore, we deduce from the conjunction of \eqref{DECGENF1}, \eqref{DECGENF2} and \eqref{DECGENF3} that for all $t \in \dR$ and for all $z \in \dC$,
$G(t,z)= \Phi(t,z) +pz \Phi(t,z) G(t,z)$,
that is
\begin{equation}
\label{DECGENFIN}
G(t,z)(1-pz \Phi(t,z))= \Phi(t,z).
\end{equation}
Denote by $z^*(t)$ the unique positive solution of the equation $pz \phi(t,z)=1$. Then, we have from \eqref{DECGENFIN} that $z^*(t)$ is a simple pole for $G$. Moreover,
for all $t \in \dR$ and for all $z \in \dC$ such that $|z| < z^*(t)$,
\begin{equation*}
G(t,z)= \frac{\Phi(t,z)}{1-pz \Phi(t,z)}.
\end{equation*}
\end{proof}

\noindent
{\bf Proof of Theorem \ref{T-LDPBERF}.} It follows from Lemma \ref{L-GEN} that the function $G(t,z)$ is analytic in $z$, for all $t \in \dR$ and for all $z \in \dC$ such that
$|z| < R_G(t)$. Moreover, one can observe from \eqref{DEFGENF} that $m_{n+2}(t)$ is the coefficient of the Taylor expansion of $G(t,z)$ at point $(t,0)$. Then, we deduce from
\eqref{CALCULATIONGEN} and the Taylor expansion of $G(t,z)$ at the neighborhood of $z^{*}(t)$ that
\begin{equation}
\label{DEVMNT}
m_n(t)=C(t) \big(\lambda(t)\big)^n \big(1+r_n(t)\big)
\qquad \text{with} \qquad \lambda(t)=\frac{1}{z^*(t)}
\end{equation}
where
$$
C(t)=\frac{z^*(t)\Phi(t,z^*(t))}{p\big(\Phi(t, z^*(t)) +z^*(t) \Phi^\prime(t,z^*(t))}
$$
and the remainder $r_n(t)$ goes to zero as $n$ tends to infinity. Hereafter,
let $\Lambda_n$ be the normalized cumulant generating function defined, for all $t \in \dR$ and for all $n \geq 1$, by
\begin{equation}
\label{DEFLAMBDANT}
\Lambda_n(p,t)= \frac{1}{n} \log \big( m_n(t) \big).
\end{equation}
We immediately obtain from \eqref{DEVMNT} that for all $t \in \dR$,
\begin{equation}
\label{LIMLAMBDANT}
\lim_{n \rightarrow \infty} \Lambda_n(p,t)=\Lambda(p,t)
\end{equation}
where $\Lambda(p,t)=\log\big(\lambda(t)\big)$. It exactly means that the limiting function $\Lambda(p,t)$ satisfies the equation
$p \exp\big( -\Lambda(p,t) \big)\Phi\big(t,\exp\big( -\Lambda(p,t) \big)\big)=1$, which can be rewritten as
\begin{equation}
\label{PROOFLAMBDA}
\dE\Big[ F^t_{G+3} \! \exp\big(\!-\!(G+1)\Lambda(p,t)\big) \Big] 
=1.
\end{equation}
where $G$ stands for a random variable with Geometric $\cG(p)$ distribution. The function $\Lambda(p,t)$ is finite for all $t \in \dR$, since $|\Lambda(p,t)| \leq |t| \log(2)$. It is also infinitely differentiable and strictly convex, as seen in Remark \ref{R-LDP}.
Finally, it follows from the G\"{a}rtner-Ellis theorem \cite[Theorem 2.3.6]{DZ2010} that the sequence $(\log(X_n)/n)$ satisfies a large deviation principle with speed $n$ and good rate function given by \eqref{RATEF}, which completes the proof of Theorem \ref{T-LDPBERF}.
\demend

\vspace{-1ex}
\noindent
{\bf Acknowledgments.} The first author would like to thank Jean-Fran\c{c}ois Marckert and Jean-Fran\c{c}ois Quint for fruitful discussions on a preliminary version of the manuscript.

\nocite{*}

\bibliographystyle{acm}
\bibliography{Bernoulli-Fib-2026}

\end{document}